\documentclass[12pt]{article}
\usepackage{a4wide}
\usepackage{amssymb, amsmath}
\usepackage{color}
\usepackage[utf8]{inputenc}
\usepackage[english]{babel}
\usepackage{amsthm}
\usepackage{amsfonts}
\usepackage{hyperref}
\usepackage[capitalise,noabbrev]{cleveref}
\usepackage{graphicx}
\usepackage[dvipsnames]{xcolor}
\usepackage{tikz-cd}
\usepackage{stmaryrd}
\usepackage{mathdots}
\usepackage{thmtools}
\usepackage{thm-restate}
\usepackage{subcaption}
\usepackage{authblk}
\usepackage{multicol}
\usepackage{mathrsfs}

\newcommand{\ZZ}{\mathbb{Z}}
\newcommand{\QQ}{\mathbb{Q}}
\newcommand{\RR}{\mathbb{R}}

\newcommand{\PP}{\mathbb{P}}

\newcommand{\GG}{\mathbb{G}}
\newcommand{\F}{\mathcal{F}}
\renewcommand{\O}{\mathcal{O}}
\renewcommand{\AA}{\mathbb{A}}
\newcommand{\Q}{\mathfrak{Q}}
\newcommand{\C}{\mathfrak{C}}
\newcommand{\X}{\mathcal{X}}

\newcommand{\Zpos}{\mathbb{Z}_{\geq 0}}

\renewcommand{\phi}{\varphi}
\renewcommand{\epsilon}{\varepsilon}

\newcommand{\PGL}{\mathrm{PGL}}
\renewcommand{\S}{\mathcal{S}} 

\newcommand{\Quot}{\mathrm{Quot}}
\newcommand{\QuotF}{\mathfrak{Q}uot}
\DeclareMathOperator{\Spec}{Spec}

\newcommand{\Gr}{\mathrm{Gr}}
\newcommand{\Fl}{\mathrm{Fl}}

\newcommand{\Hilb}{\mathrm{Hilb}}
\newcommand{\Hom}{\mathrm{Hom}}
\newcommand{\Sch}{\mathrm{Sch}}
\newcommand{\Sets}{\mathrm{Sets}}
\newcommand{\id}{\mathrm{id}}
\renewcommand{\tilde}[1]{\widetilde{#1}}

\newcommand{\set}[2]{\left\{ #1 \, \middle| \, #2 \right\}}

\renewcommand{\vec}[1]{\mathbf{#1}}

\newcommand{\vl}[1]{\multicolumn{1}{|c}{#1}}

\theoremstyle{plain}
\newtheorem{theorem}{Theorem}
\newtheorem{lemma}[theorem]{Lemma}

\newtheorem{corollary}[theorem]{Corollary}
\newtheorem{proposition}[theorem]{Proposition}

\theoremstyle{definition}
\newtheorem{definition}[theorem]{Definition}

\theoremstyle{remark}
\newtheorem{remark}[theorem]{Remark}
\newtheorem{example}[theorem]{Example}

\title{Murphy's law on a fixed locus of the Quot scheme}
\author{Reinier F. Schmiermann}
\date{}

\begin{document}

\maketitle

\begin{abstract}
Let $T := \mathbb{G}_m^d$ be the torus acting on the Quot scheme of points $\coprod_n \mathrm{Quot}_{\mathcal{O}^r/\mathbb{A}^d/\mathbb{Z}}^n$ via the standard action on $\mathbb{A}^d$. We analyze the fixed locus of the Quot scheme under this action. In particular we show that for $d \leq 2$ or $r \leq 2$, this locus is smooth, and that for $d \geq 4$ and $r \geq 3$ it satisfies Murphy's law as introduced by Vakil, meaning that it has arbitrarily bad singularities. These results are obtained by giving a decomposition of the fixed locus into connected components, and identifying the components with incidence schemes of subspaces of $\mathbb{P}^{r-1}$. We then obtain a characterization of the incidence schemes which occur, in terms of their graphs of incidence relations.
\end{abstract}

\section{Introduction}
Understanding the singularities of the Quot scheme of points $\Quot_{\O^r/\AA^d/\ZZ}^n$ is a topic of ongoing interest. While for $d = 1$ this moduli space is smooth, already starting from $d = 2$ it exhibits singularities. It is known that in general the Quot scheme has multiple irreducible components, some of which are even generically nonreduced \cite{jelisiejew_sivic2022components}. For an overview of some open problems in this area, see \cite{jelisiejew2023openproblems}.

A new perspective on studying singularities of moduli spaces was given by the introduction of Murphy's law in algebraic geometry by Vakil \cite{vakil2006murphy}.
\begin{definition}
Consider the equivalence relation $\sim$ on pointed schemes which is generated by setting $(X, p) \sim (Y, q)$ if there is a smooth morphism $X \to Y$ which sends $p$ to $q$. An equivalence class of this relation is called a \emph{singularity type}. Given a scheme $X$, we say that $X$ satisfies \emph{Murphy's law} if every singularity type that appears on a scheme of finite type over $\ZZ$, also appears on $X$.
\end{definition}
Vakil already proved that a large number of common moduli spaces satisfy Murphy's law. However, the Quot scheme of points was not one of them. In an influential paper by Jelisiejew \cite{jelisiejew2020pathologies}, he proved that the Hilbert scheme of points on $\AA_\ZZ^{16}$ satisfies Murphy's law up to retraction. It seems unlikely that 16 is the smallest dimension for which this behavior occurs, so this raises the question what the minimal dimension $d$ is such that the Hilbert scheme on $\AA^d$ satisfies Murphy's law. More generally, we might ask for what values of $d$ and $r$ the Quot scheme $\coprod_n \Quot_{\O^r/\AA^d/\ZZ}^n$ satisfies Murphy's law (possibly up to retraction).

As a first step towards answering this question, we might consider some torus action on the Quot scheme, and start by analyzing the fixed locus of this action. In particular, there is a natural $\GG_m^d \times \GG_m^r$ action on $\Quot_{\O^r/\AA^d/\ZZ}^n$, where the action of $\GG_m^d$ comes from the standard action on $\AA^d$, and the subtorus $\GG_m^r$ acts on the framing bundle $\O^r$ by scaling the summands. It is well-known that the fixed locus under the action of this large torus consists of a finite number of reduced points, but by looking at several interesting subtori, we can get fixed loci with a more interesting structure. For example, Erman \cite{erman2012murphy} considered the action of the one dimensional torus $\GG_m$ on $\Hilb_{\AA^d}^n$ acting as the diagonal of $\GG_m^d$, and showed that the corresponding fixed locus satisfies Murphy's law for $d \geq 5$. This result plays an important role in the proof of Jelisiejew's result on the full Hilbert scheme. We also mention work by Bifet \cite{bifet1989pointfixes}, who studied the fixed locus of $\Quot_{\O^r/\AA^d/k}^n$ under the action of $\GG_m^r$, and Payne \cite{payne2008moduli}, who showed Murphy's law on certain moduli of toric vector bundles on toric varieties.

Another motivation for studying various fixed loci of the Quot scheme comes from enumerative geometry. When defining and computing virtual invariants on the Quot scheme on $\AA^d$, often torus localization is used to express the invariants in terms of a more manageable fixed locus. This has been done for $\AA^2$ \cite{bojkohuang2023invariants} (See also \cite{oprea2022quot, arbesfeld2021Ktheory, stark2022cosection}), $\AA^3$ \cite{fasola2021higher} and $\AA^4$ \cite{nekrasov2019magnificent, KRmagnificent}.

In this paper, we study the fixed locus of $\Quot_{\O^r/\AA^d/\ZZ}^n$ under the action of the torus $T := \GG_m^d$. Our main result is as follows:
\begin{theorem}
Let $d$ and $r$ be nonnegative integers. Then the fixed locus
\[
\coprod_{n \geq 0} \left(\Quot_{\O^r/\AA^d/\ZZ}^n\right)^T
\]
is smooth if and only if $d \leq 2$ or $r \leq 2$. Furthermore, if $d \geq 4$ and $r \geq 3$, then this scheme satisfies Murphy's law.
\end{theorem}
Note that we work over $\Spec \ZZ$ here, the main reason for this being that this is the setting in which Murphy's law is defined. It should however be noted that by base changing, the smoothness result holds over any field. Furthermore, note that if $X$ is any $d$-dimensional toric variety which contains a copy of $\AA^d$ (in particular if $X = \PP^d$), then $\Quot_{\O^r/X/\ZZ}^n$ contains a copy of $\Quot_{\O^r/\AA^d/\ZZ}^n$ as an open subscheme, so Murphy's law also holds for the Quot scheme on $X$ if $d \geq 4$ and $r \geq 3$.

For $d = 2$, note that the theorem tells us that the fixed locus $(\Quot_{\O^r/\AA^2/\ZZ}^n)^T$ is smooth for all $r$ and $n$. This may be unexpected, since $\Quot_{\O^r/\AA^2/\ZZ}^n$ itself is singular for $r \geq 2$ and $n \geq 2$ \cite[Remark 4.2]{oprea2022quot}.

The first case in which we show Murhpy's law, is the case $d = 4, r = 3$. Note that, somewhat surprisingly, it is still possible to define virtual enumerative invariants on the Quot scheme on $\AA^4$, as mentioned before. This suggests that even moduli spaces that satisfy Murphy's law can carry additional interesting structure which can be studied. A proof for Murphy's law in just the case $d = 4, r = 3$ is given in the author's Master's thesis \cite{schmiermann2021murphy}, on which this paper is based. We also note that this case can be used to show that the moduli space of Pandharipande-Thomas stable pairs on a toric Calabi-Yau fourfold has arbitrarily bad singularities, as described by Liu \cite[Example 2.2.5]{liu2023PT}.

The cases with $d = 3$ and $r \geq 3$ are the only cases where we do not know whether Murphy's law holds, even though we do know that the fixed locus is singular (See \cref{ex:dim3sing}). It would be interesting to know more about the singularities which occur in this case.

\subsection{Proof overview}
To prove the theorem, we start by giving a decomposition of the fixed locus into connected components, based on what we will call \emph{characteristic functions}. This decomposition is similar to the decompositions of moduli spaces of torus equivariant sheaves given in \cite{kool2011fixed} or \cite{payne2008moduli}. Next we show that the components in this decomposition are isomorphic to incidence schemes of subspaces of $\PP^{r-1}$. These incidence schemes parametrize configurations of subspaces (such as points, lines, planes, etc.) in $\PP^{r-1}$ satisfying certain incidence relations (such as ``line $i$ is contained in plane $j$'').

Using this identification of components of the fixed locus with incidence schemes, the question of which incidence schemes occur as components becomes a mostly combinatorial problem. We show that the relevant incidence schemes are exactly those that correspond to certain intersection graphs of connected subsets in $\RR^{d-1}$. When $d = 2$ this means we get incidence schemes corresponding to interval graphs, which we show to be smooth. In the case $d = 3$, the relevant incidence schemes correspond to string graphs. Finally, in the case $d = 4$, we show that all possible incidence structures occur. This essentially follows from the fact that all graphs can be embedded in $\RR^3$. Now we can apply Mn\"ev-Sturmfels universality \cite{mnev2006universality,sturmfels1987decidability} (see also \cite{lafforgue2003chirurgie,lee2013mnevsturmfels} for more modern scheme-theoretic formulations), which states that the disjoint union of all incidence schemes of points and lines in $\PP^2$ satisfies Murphy's law.

In \cref{sec:quot}, we recall the definition of the Quot scheme and the torus action which we use. In \cref{subsec:decomp}, the decomposition of the fixed locus is described. \Cref{sec:incidence} starts with a definition of incidence schemes, after which it is shown that every component of the decomposition is isomorphic to such an incidence scheme. The characterization of the incidence schemes occurring in the decomposition is given in \cref{sec:char}. Finally, in \cref{sec:appl} the main result is proven. This happens by separately considering the cases with small rank (\cref{subsec:rank12}), dimension 1 or 2 (\cref{subsec:dim12}), dimension 3 (\cref{subsec:dim3}) and finally dimension 4 and higher (\cref{subsec:dim4}).

\subsection{Acknowledgments}
I would like to thank Joachim Jelisiejew, Henry Liu and Michele Graffeo for useful discussions. Furthermore, I would like to thank my supervisor Martijn Kool for his help and guidance. This work was supported by NWO grant VI.Vidi.192.012.

\subsection{Notation and conventions}
For a noetherian scheme $S$, we denote by $\Sch_S$ the category of locally noetherian schemes over $S$. If $S = \Spec R$, we will also denote this category by $\Sch_R$. For $T, X \in \Sch_S$, we denote $X(T) := \Hom(T, X)$ for the set of $T$-points on $X$, and $X_T = X \times_S T$ for the base-change of $X$ to $T$. If furthermore $\F$ is a sheaf on $X$, the we write $\F_T$ for the pull-back of $\F$ under the projection $X_T \to X$, and for $f \colon X \to Y$ a morphism in $\Sch_S$, we write by $f_T \colon X_T \to Y_T$ the base-changed version of the map.

In case no base-scheme is specified, we work over $\Spec \ZZ$, in particular we write $\AA^d := \AA_\ZZ^d$ and $\PP^d := \PP_\ZZ^d$ for any nonnegative integer $d$.

Any graphs occurring in this paper will be assumed to be finite.

\section{Quot scheme and torus action} \label{sec:quot}

\subsection{Quot scheme}
We start by recalling the definition of the Quot scheme. The following is based on \cite[Chapter 5]{FGA_explained}, where also a proof of representability is given.

\begin{definition}
Let $S$ be a noetherian scheme, let $X$ be a quasi-projective scheme over $S$, and let $E$ be a coherent sheaf on $X$. Denote by $\Sch_S$ the category of locally noetherian schemes over $S$. For $T \in \Sch_S$, a \emph{family of quotients of $E$ parametrized by $T$} is a pair $(\F, q)$ consisting of
\begin{itemize}
\item a coherent sheaf $\F$ on $X_T = X \times_S T$ such that the schematic support of $\F$ is proper over $T$ and $\F$ is flat over $T$, and
\item a surjective homomorphism of $\O_{X_T}$-modules $q \colon E_T \to \F$, where $E_T$ is the pull-back of $E$ under the projection $X_T \to X$.
\end{itemize}
Two such families $(\F, q)$ and $(\F', q')$ are considered equivalent if $\ker(q) = \ker(q')$. 

The \emph{Quot-functor}
\[
\QuotF_{E/X/S} \colon \Sch_S^{op} \to \Sets
\]
sends any $T \in \Sch_S$ to the set of equivalence classes of families of quotients of $E$ parametrized by $T$. For a morphism $f \colon T \to T'$ in $\Sch_S$, the morphism
\[
\QuotF_{E/X/S}(f) \colon \QuotF_{E/X/S}(T') \to \QuotF_{E/X/S}(T)
\]
is given by pull-back along $f$.

Let $L$ be a line bundle on $X$ which is relatively very ample over $S$, then $\QuotF(f)$ can be written as a disjoint union
\[
\QuotF_{E/X/S} = \coprod_{\Phi \in \QQ[\lambda]} \QuotF_{E/X/S}^{\Phi, L},
\]
where the functor $\QuotF_{E/X/S}^{\Phi, L}$ maps the scheme $T$ to classes of families $(\F, q)$ such that for every $t \in T$, the Hilbert polynomial of $\F_t$ is $\Phi$.

For any $\Phi \in \QQ[\lambda]$, the functor $\QuotF_{E/X/S}^{\Phi, L}$ is representable by the \emph{Quot scheme} $\Quot_{E/X/S}^{\Phi, L}$, which is a quasi-projective scheme over $S$.
\end{definition}

In this paper, we are only interested in Quot schemes on $\AA^d$ of 0-dimensional quotients of $\O_{\AA^d}^r$, which will be denoted by $Q_{r, d}^n := \Quot_{\O_{\AA^d}^r/\AA^d/\Spec \ZZ}^{n, L}$ for nonnegative integers $d$, $r$ and $n$. Note that this scheme does not depend on the choice of a line bundle $L$, as reflected by the notation. We denote the corresponding moduli functor by $\Q_{r, d}^n$, and use the notation $Q_{r,d}^\bullet := \coprod_{n \geq 0} Q_{r,d}^n$ and $\Q_{r,d}^\bullet := \coprod_{n \geq 0} \Q_{r,d}^n$.

\subsection{Torus action}
The Quot scheme $Q_{r,d}^n$ carries an action of a $d$-dimensional torus $T = \GG_m^d$, induced by the standard action on $\AA^d$, which we will describe next.

Identify $\AA^n = \Spec \ZZ[x_1, \ldots, x_d]$ and $T = \Spec \ZZ[t_1,\ldots, t_d, t_1^{-1}, \ldots, t_d^{-1}]$. Now the action on $\AA^d$ is given by a map $a \colon T \times \AA^d \to \AA^d$, coming from the map
\[
\ZZ[x_1, \ldots, x_d] \to \ZZ[t_1,\ldots, t_d, t_1^{-1}, \ldots, t_d^{-1}] \otimes_\ZZ \ZZ[x_1, \ldots, x_d],
\]
given by $x_i \mapsto t_ix_i$ for $1 \leq i \leq d$.

Let $S \in \Sch_\ZZ$, and let $\vec{t} \in T(S)$. This $\vec{t}$ induces an isomorphism $a_{\vec{t}} \colon \AA_S^d \to \AA_S^d$, which gives a functor $a_{\vec{t} *} = a_{\vec{t}^{-1}}^*$ on the category of coherent sheaves on $\AA_S^d$. Using this functor, we can send any quotient $\O_{\AA_S^d}^r \xrightarrow{q} \F$ in $\Q_{r,d}^n(S)$ to another quotient
\[
\O_{\AA_S^d}^r \cong a_{\vec{t} *} \O_{\AA_S^d}^r \xrightarrow{a_{\vec{t} *} q} a_{\vec{t} *} \F.
\]
Varying the base $S$ and the point $\vec{t}$, this yields a natural transformation $T \times \Q_{r,d}^n \to \Q_{r,d}^n$. It can be verified that this gives a $T$-action on $\Q_{r,d}^n$ and $Q_{r,d}^n$.

Now that we have defined a group action, we can also consider the fixed point locus $(Q_{r,d}^n)^T$, or the corresponding functor $(\Q_{r,d}^n)^T$, see \cite{fogarty1973} for how these can be defined. For a scheme $S$, the set $(\Q_{r,d}^n)^T(S)$ contains exactly those points $(\F, q) \in \Q_{r,d}^n(S)$ such that for every scheme $S'$ over $S$ we have that $(\F_{S'}, q_{S'})$ is fixed under the action of $T(S')$.

\begin{definition}
Let $X$ be a scheme over a base scheme $S \in \Sch_\ZZ$. Let $a \colon G \times_S X$ be an action of a group scheme $G$ on $X$, and let $\F$ be a sheaf on $X$. Then a \emph{$G$-equivariant structure} on $\F$ is an isomorphism $\alpha \colon a^* \F \to \pi_2^* \F$ of sheaves on $G \times_S X$, with $\pi_2 \colon G \times X \to X$ the projection, such that the following diagram of sheaves on $G\times G\times X$ commutes:
\[
\begin{tikzcd}
(\id_G \times a)^*a^*\F \ar[r,"\sim"] \ar[d,"(\id_G\times a)^*\alpha"] & (m \times \id_X)^*a^*\F \ar[d,"(m\times\id_X)^*\alpha"]\\
(\id_G \times a)^*\pi_2^*\F \ar[r,"\pi_{23}^*\alpha"] & \pi_3^*\F
\end{tikzcd}
\]
where $\pi_2 \colon G \times X \to X$, $\pi_{23} \colon G\times G\times X \to G \times X$ and $\pi_3 \colon G \times G \times X \to X$ are projections.
\end{definition}

\begin{lemma} \label{lemma:fixed_to_equivariant}
For $S \in \Sch_\ZZ$ and $(\F,q) \in \Q_{r,d}^n(S)$, we have that $(\F,q) \in (\Q_{r,d}^n)^T(S)$ if and only if $\F$ can be given a $T$-equivariant structure such that $q\colon \O_{\AA_S^d}^r \to \F$ is map of $T$-equivariant sheaves. Furthermore, if $\F$ has such a $T$-equivariant structure, then this structure is unique.
\end{lemma}
\begin{proof}
First, assume that $(\F, q) \in (\Q_{r,d}^n)^T(S)$. The identity map $\id_{T_S} \colon T_S \to T_S$ gives a point $\vec{t} \in T_S(T_S)$. It can be verified that this $\vec{t}$ acts on $\AA_{T_S}^d \cong T_S \times_S \AA_S^d$ using the composition
\[
a_{\vec{t}} \colon T_S \times_S \AA_S^d \xrightarrow{\Delta \times \id_{\AA_S^d}} T_S \times_S T_S \times_S \AA_S^d \xrightarrow{\id_{T_S} \times a_S} T_S \times_S \AA_S^d.
\]
In particular, we have that $\pi_2 \circ a_{\vec{t}} = a_S$, where $\pi_2 \colon T_S\times_S \AA_S^d \to \AA_S^d$ is the projection and $a_S \colon T_S \times_S \AA_S^d \to \AA_S^d$ is the base-changed version of $a \colon T \times \AA^d \to \AA^d$.

Note that the map $\Q_{r,d}^n(S) \to \Q_{r,d}^n(T_S)$ sends $(\F, q)$ to $(\pi_2^* \F, \pi_2^* q)$. The action of $\vec{t}^{-1}$ sends this point to $(a_{\vec{t}}^* \pi_2^* \F, a_{\vec{t}}^* \pi_2^* q) = (a_S^* \F, a_S^* q)$. Applying the fact that $(\F, q)$ is a $T$-fixed point, we get that $(\pi_2^* \F, \pi_2^* q)$ and $(a_S^* \F, a_S^* q)$ represent the same element of $\Q_{r,d}^n(T_S)$, so we get $\ker(\pi_2^* q) = \ker(a_S^* q)$. From this we get an isomorphism $\alpha \colon a_S^*\F \to \pi_2^*\F$ which fits in the following commutative diagram:
\[
\begin{tikzcd}
\ker(a_S^* q) \ar[r,"\sim"]\ar[hook,d] & \ker(\pi_2^* q)\ar[hook,d]\\
a_S^* \O_{\AA_S^d}^r \ar[r,"\sim"]\ar[two heads,d,"a_S^* q"]& \pi_2^* \O_{\AA_S^d}^r\ar[two heads,d,"\pi_2^* q"]\\
a_S^*\F \ar[r,"\alpha"]& \pi_2^*\F\\
\end{tikzcd}
\]
Here the first two rows come from the equivariant structure on $\O_{\AA_S^d}^r$. Using the fact that $q$ is a surjection which sends the equivariant structure on $\O_{\AA_S^d}^r$ to $\alpha$, it follows that also $\alpha$ describes an equivariant structure, and furthermore that this $\alpha$ is unique if we require this compatibility.

Now assume that $\F$ has an equivariant structure $\alpha$ such that $q$ is a morphism of $T$-equivariant sheaves. Now let $S' \to S$ be an $S$-scheme, take $\vec{t} \in T(S')$, and let $i \colon S' \to T_{S'}$ be the morphism of $S$-schemes corresponding to $\vec{t}^{-1}$. We get the following commuting diagram of sheaves on $\AA_{S'}^d$:
\[
\begin{tikzcd}
(i\times\id)^*a_{S'}^*\O_{\AA_{S'}^d}^r \cong a_{\vec{t}^{-1}}^*\O_{\AA_{S'}^d}^r \ar[r,"\sim"]\ar[two heads,d,"a_{\vec{t}^{-1}}^* q_{S'}"] & (i\times\id)^*\pi_2^*\O_{\AA_{S'}^d}^r \cong \O_{\AA_{S'}^d}^r \ar[two heads,d,"q_{S'}"]\\
(i\times\id)^*a_{S'}^*\F_{S'} \cong a_{\vec{t}^{-1}}^*\F_{S'} \ar[r,"\sim"] & (i\times\id)^*\pi_2^*\F_{S'} \cong \F_{S'}\\
\end{tikzcd}
\]
From this it becomes clear that the family $(\F_{S'}, q) \in \Q_{r,d}^n(S')$ is fixed under the action of $\vec{t}$. Since this holds for all $S'$ and all $t \in T(S')$, it follows that $(\F, q) \in (\Q_{r,d}^n)^T$.
\end{proof}

\subsection{Decomposing the fixed locus} \label{subsec:decomp}

We know by \cite[Lemma 0EKL]{stacks-project} that over an affine scheme $\Spec R$, $T$-equivariant quasi-coherent sheaves on $\AA_{\Spec R}^d$ correspond to $\ZZ^d$-graded $R[x_1,\ldots,x_d]$-modules. In particular, the points in $(\Q_{r,d}^n)^T(\Spec R)$ correspond exactly to graded quotient modules of the form
\[
q \colon R[x_1, \ldots, x_d]^r \to F = \bigoplus_{\vec{a} \in \ZZ^d} F_\vec{a}
\]
that are furthermore flat over $R$. Here we take the convention that $\prod_{i = 1}^d x_i^{a_i}$ has degree $(a_i)_{i=1}^d \in \Zpos^d$.

In order to describe the scheme $(Q_{r,d}^n)^T$ in more detail, we introduce characteristic functions. An example of this definition is visualized in \cref{fig:chi}.

\begin{figure}
\centering
\begin{subfigure}[t]{0.68\textwidth}
\centering
\begin{tabular}{|ccccc}
$\vdots$                      & $\vdots$                           & $\vdots$                           & $\vdots$ &  $\iddots$\\
0                                  & 0                                  & 0                                  & 0 & $\cdots$\\
$(k^3/\langle e_1,e_2\rangle)y^2$ & 0                                  & 0                                  & 0 & $\cdots$\\
$(k^3/\langle e_1\rangle)y$   &  $(k^3/\langle e_1,e_3\rangle)xy$ & 0                                  & 0 & $\cdots$\\
$k^3$                         &  $(k^3/\langle e_3\rangle)x$  & $(k^3/\langle e_3\rangle)x^2$  & 0 & $\cdots$\\\hline
\end{tabular}
\caption{The vector spaces $F_\vec{a}$} \label{fig:chi_Fa}
\end{subfigure}
\begin{subfigure}[t]{0.3\textwidth}
\centering
\begin{tabular}{|ccccc}
$\vdots$ & $\vdots$ & $\vdots$ & $\vdots$ & $\iddots$\\
0 & 0 & 0 & 0 & $\cdots$\\
1 & 0 & 0 & 0 & $\cdots$\\
2 & 1 & 0 & 0 & $\cdots$\\
3 & 2 & 2 & 0 & $\cdots$\\\hline
\end{tabular}
\caption{The corresponding $\chi_\F$} \label{fig:chi_chi} 
\end{subfigure}
\caption{A visualization of the characteristic function $\chi_\F \colon \Zpos^2 \to \ZZ^3$ of the quotient $\O_{\AA_k^2}^3 \to \F$ on $\AA_k^2$, corresponding to $F = (k[x,y])^3/((x,y)^3 + (e_3x, e_1y, e_2y^2))$.}
\label{fig:chi}
\end{figure}

\begin{definition}
Let $k$ be a field and let $F$ be a graded quotient module which corresponds to some point $(\F, q)$ in $(\Q_{r,d}^n)^T(\Spec k)$. We write $F = \bigoplus_{\vec{a} \in \Zpos^d} F_\vec{a}$. Each of the components $F_\vec{a}$ is a finite-dimensional $k$-vector space, so we can define the \emph{characteristic function} of $\F$ to be $\chi_\F\colon \Zpos^d \to \ZZ$ with $\chi_\F(\vec{a}) = \dim_k F_{\vec{a}}$ for all $\vec{a} \in \Zpos^d$.

For general $S \in \Sch_\ZZ$ and $(\F, q) \in (\Q_{r,d}^n)^T(S)$, we can define for every point $s \in S$ the characteristic function $\chi_{\F_s}$ by considering the fiber $(\F_s, q_s) \in (\Q_{r,d}^n)^T(\Spec k(s))$.
\end{definition}

We define $\X_{r,d}^n$ to be the set of all possible characteristic functions $\chi_{\F_s}$ for any scheme $S$ and any point $s \in S$. Furthermore, for $\vec{a} = (a_i)_{i=0}^d$ and $\vec{b} = (b_i)_{i=0}^d$ in $\Zpos^d$ we say that $\vec{a} \leq \vec{b}$ if and only if $a_i \leq b_i$ for all $i$.

\begin{lemma}\label{lemma:X}
A function $\chi \colon \Zpos^d \to \ZZ$ is contained in $\X_{r,d}^n$ if and only if the following properties are satisfied:
\begin{enumerate}
\item For all $\vec{a} \in \Zpos^d$, we have $0 \leq \chi(\vec{a}) \leq r$,
\item $\sum_{\vec{a} \in \Zpos^d} \chi(\vec{a}) = n$,
\item For all $\vec{a}, \vec{b} \in \Zpos^d$ with $\vec{a} \leq \vec{b}$, we have $\chi(\vec{a}) \geq \chi(\vec{b})$.
\end{enumerate}
\end{lemma}
\begin{proof}
First suppose that $\chi \in \X_{r,d}^n$, so there is some field $k$ and some $\ZZ^d$-graded quotient $q \colon k[x_1, \ldots, x_d]^r \to F$ with
\[
\chi(\vec{a}) = \dim_k F_{\vec{a}} = r - \dim_k (\ker q)_\vec{a}
\]
for all $\vec{a} \in \ZZ^d$. Note that $\ker q$ will always be a graded torsion-free submodule of $k[x_1,\ldots, x_d]^r$. Properties 1 and 3 follow immediately from this. Property 2 follows from the fact that the sheaf $\F$ should have constant Hilbert polynomial $n$.

For the other direction, fix some field $k$ and some sequence of vector spaces
\[
0 = V_0 \subseteq V_1 \subseteq \cdots \subseteq V_r = k^r
\]
such that $\dim_k V_i = i$ for $0 \leq i \leq r$. We consider the graded submodule
\[
E = \bigoplus_{\vec{a}\in\Zpos^d} E_{\vec{a}} \subseteq k[x_1, \ldots, x_d]^r
\]
with $E_\vec{a} = V_{r-\chi(\vec{a})} \cdot x^{\vec{a}}$ for all $\vec{a} \in \ZZ^d$. Now the quotient $k[x_1, \ldots, x_d] \to k[x_1, \ldots, x_d]/E$ is well-defined and corresponds to some element of $(\Q_{r,d}^n)^T(\Spec k)$ which furthermore has characteristic function $\chi$.
\end{proof}

\begin{lemma}
Let $S \in \Sch_{\Spec \ZZ}$ and let $(\F, q) \in (\Q_{r,d}^n)^T(S)$. Then the characteristic function $\chi_{\F_s} \in \X_{r,d}^n$ is locally constant over $S$.
\end{lemma}
\begin{proof}
Let $U = \Spec R \subseteq S$ be an affine, Noetherian open. It is sufficient to prove that $\chi_{\F_s}$ is locally constant on $U$. On $U$, we can identify $\F$ with a graded quotient module
\[
R[x_1, \ldots, x_d]^r \to F = \bigoplus_{\vec{a} \in \Zpos^d} F_\vec{a}.
\]
By flatness of $F$ over $R$, it follows that each $F_{\vec{a}}$ is flat. By the fact that $U$ is Noetherian and this flatness it follows (see \cite[Lemma 00NX]{stacks-project}) that $\dim_{k(s)} F_{\vec{a}} \otimes_R k(s)$ is locally constant when $s$ varies through $U = \Spec R$. This immediately implies that also $\chi_{\F_s}$ is locally constant.
\end{proof}

\begin{corollary}
We have decompositions
\[
(\Q_{r,d}^n)^T = \coprod_{\chi \in \X_{r,d}^n} \Q_\chi
\]
and
\[
(Q_{r,d}^n)^T = \coprod_{\chi \in \X_{r,d}^n} Q_\chi,
\]
where for $\chi \in \X_{r,d}^n$ we have that $\Q_\chi$ is the subfunctor of $(\Q_{r,d}^n)^T$ of families $(\F, q) \in (\Q_{r,d}^n)^T(S)$ which have characteristic function $\chi$ over every point $s \in S$, and $Q_\chi$ is a scheme representing $\Q_\chi$.
\end{corollary}
\begin{proof}
This is a direct consequence of the previous lemma.
\end{proof}

\begin{remark}
Note that the parameters $d$ and $n$ can be recovered from the characteristic function $\chi$ by \cref{lemma:X}. Technically, the scheme $Q_\chi$ and its functor $\Q_\chi$ may depend on $r$, but in what follows we will always assume the used value of $r$ to be clear from the context.
\end{remark}

\section{Incidence schemes} \label{sec:incidence}
In order to describe the schemes $Q_\chi$ for $\chi \in \X_{r,d}^n$, we will relate them to incidence schemes of subspaces of $\PP^{r-1}$. In what follows, we will denote by $\Gr(m,r)$ the Grassmannian $m$-dimensional subspaces of $\AA^r$, or equivalently of $(m-1)$-dimensional subspaces of $\PP^{r-1}$. More explicitly, the functor of points of $\Gr(m,r)$ assigns to a scheme $S$ the set of quotients $\O_S^r \to \F$ where $\F$ is a locally free sheaf of rank $r-m$.

Similarly, for $m_1 < \cdots < m_k \leq r$ nonnegative integers, we denote by $\Fl(m_1, \ldots, m_k, r)$ the flag variety parametrizing length $k$ flags of subspaces $E_1 \subseteq \cdots \subseteq E_k \subseteq \PP^k$ with $\dim E_i = m_i-1$. This is a smooth, closed subscheme of $\prod_{i=1}^k \Gr(m_i, r)$.

The definition of an incidence structure given here is based on that from \cite[Section 5.1.1]{configurations}.
\begin{definition}
For $k \in \Zpos$, a \emph{rank $k$ incidence structure} is a tuple $\S = (P_1, \ldots, P_k, I)$ where the $P_i$ are disjoint index sets, and $I \subseteq \bigcup_{i < j} P_i \times P_j$. For such a structure, we will denote $P = \bigcup_i P_i$ and for $j \in P_i$, we write $d(j) = i$. Given two structures $\S_1 = (P_1, \ldots, P_k, I_1)$ and $\S_2 = (P_1', \ldots, P_k', I_2)$, we will say that they are \emph{equivalent} if there are identifications $P_i \cong P_i'$ for all $i$, such that $I_1$ and $I_2$ have the same transitive closure, when seen as antisymmetric relations on $\bigcup_i P_i \cong \bigcup_i P_i'$.

Given an incidence structure $\S$, we define a corresponding functor
\[
\C_\S \colon \Sch_\ZZ^{op} \to \Sets
\]
which parametrizes (families of) collections $(E_i)_{i \in P}$, where $E_i$ for $i \in P$ is an $d(i)-1$ dimensional subspace of $\PP^k$, such that the $E_i \subseteq E_j$ whenever $(i, j) \in I$. More formally, we define $\C_\S(S)$ for $S \in \Sch_\ZZ$ to be the subset of $\prod_{i \in P} \Gr(d(i), k+1)(S)$ consisting of tuples $(\F_i, q_i)_{i \in P}$ satisfying
\begin{itemize}
	\item $\F_i$ is a rank $k+1-d(i)$ locally free sheaf on $S$ for all $i \in P$,
	\item the maps $q_i \colon \O_S^r \to \F_i$ are surjective for all $i \in P$,
	\item $\ker(q_i) \subseteq \ker(q_j)$ for all $(i, j) \in I$.
\end{itemize}
Let the \emph{incidence scheme} $C_\S$ be the scheme representing this functor. This scheme can be constructed using a Cartesian square of the form
\[
\begin{tikzcd}
C_\S \ar[r] \ar[d] & \prod_{i \in P} \Gr(d(i), k+1) \ar[d]\\
\prod_{(i,j) \in I} \Fl(d(i), d(j), k+1) \ar[r,hook] & \prod_{(i,j) \in I} \Gr(d(i), k+1) \times \Gr(d(j), k+1)
\end{tikzcd}
\]
Note that equivalent incidence structures have the same functor $\C_\S$, so they also induce isomorphic incidence schemes.
\end{definition}

\begin{remark}
Often, when defining an incidence scheme, also conditions of the form ``if $(i, j) \not\in I$, then $E_i \not\subseteq E_j$'' are imposed, and a $\PGL(k+1)$ quotient is taken. (see e.g.~\cite{lafforgue2003chirurgie,lee2013mnevsturmfels}). We do not do this here.
\end{remark}

\begin{remark}
Note that a rank $k$ incidence structure also uniquely corresponds to an $k$-partite graph: here the vertices are given by $P$, and the edges are given by $I$. Here we consider a $k$-partite graph to be an undirected graph $(V, E)$, together with a partition $V = V_1 \cup \cdots \cup V_k$ such that for every edge $\{v, w\} \in E$, we have that $v$ and $w$ are contained in different parts of the partition. We will use this identification between incidence structures and $k$-partite graphs later.
\end{remark}

In the remainder of this section, we will prove that each of the schemes $Q_\chi$ is isomorphic to some incidence scheme. First we introduce some terminology for working with elements and subsets of $\Zpos^d$.

\begin{definition}
Let $\vec{a} = (a_i)_{i=1}^d, \vec{b} = (b_i)_{i=1}^d \in \Zpos^d$. We say that $\vec{a}$ and $\vec{b}$ are \emph{adjacent} if there is some $1 \leq j \leq d$ such that $a_i = b_i$ for all $i \neq j$ and $a_j = b_j \pm 1$. We write $\vec{a} \leq \vec{b}$ if $a_i \leq b_i$ for all $1 \leq i \leq d$. We say that a subset $A \subseteq \Zpos^d$ is \emph{connected} if for any $\vec{a}, \vec{b} \in A$ there are $\vec{a} = \vec{a}_0, \vec{a}_1, \ldots, \vec{a}_k = \vec{b} \in A$ for some $k \in \Zpos$ such that $\vec{a}_i$ and $\vec{a}_{i+1}$ are adjacent for all $0 \leq i \leq k-1$. For each $A \subseteq \Zpos^d$, we define its \emph{connected components} to be the maximal connected subsets of $A$.
\end{definition}

\begin{figure}
\centering
\begin{subfigure}[m]{0.49\textwidth}
\centering
\begin{tabular}{|ccccccc}
$\vdots$ & $\vdots$ & $\vdots$ & $\vdots$ & $\vdots$ & $\vdots$ & $\iddots$\\
0 & 0    & 0    & 0    & 0    & 0    & $\cdots$\\\cline{1-3}
2 & \vl1 & 1    & \vl0 & 0    & 0    & $\cdots$\\\cline{1-3}
3 & 3    & \vl2 & \vl0 & 0    & 0    & $\cdots$\\\cline{1-1}\cline{4-5}
4 & \vl3 & \vl2 & 2    & \vl1 & \vl0 & $\cdots$\\\cline{2-3}
4 & 4    & \vl3 & \vl2 & \vl1 & \vl0 & $\cdots$\\\cline{3-5}
4 & 4    & 4    & 4    & \vl3 & \vl0 & $\cdots$\\\hline
\end{tabular}
\caption{A characteristic function $\chi \in \X_{4,2}^\bullet$} \label{fig:incidence_chi}
\end{subfigure}
\begin{subfigure}[m]{0.49\textwidth}
\centering
\includegraphics[width=0.7\textwidth]{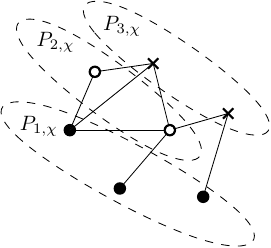}
\caption{The 3-partite graph corresponding to $\S_\chi$} \label{fig:incidence_S} 
\end{subfigure}
\caption{A characteristic function $\chi$ and the corresponding incidence structure $\S_\chi$.}
\label{fig:incidence}
\end{figure}

Fix some $\chi \in \X_{r,d}^n$. We will construct a rank $r-1$ incidence structure $\S_\chi$ such that $Q_\chi \cong C_{\S_{\chi}}$. An example of this construction is shown in \cref{fig:incidence}. For each $1 \leq i \leq r-1$, let $P_{i, \chi}$ be the set of connected components of $\chi^{-1}(r-i) \subseteq \Zpos^d$. For $\vec{a} \in \Zpos^d$ with $\chi(\vec{a}) = r-i$, denote by $[\vec{a}] \in P_{i, \chi}$ the connected component that contains $\vec{a}$. Define $I_\chi$ to contain exactly those pairs $(p, q) \in \bigcup_{i < j} P_{i,\chi} \times P_{j,\chi}$ for which there are $\vec{a} \in p$ and $\vec{b} \in q$ such that $\vec{a}$ and $\vec{b}$ are adjacent and $\vec{a} \leq \vec{b}$. We take $\S_\chi$ to be the incidence structure $(P_{1,\chi}, \ldots, P_{r-1,\chi}, I_\chi)$.

\begin{lemma} \label{lemma:Q_is_C}
For all $\chi \in \X_{r,d}^n$, there is an isomorphism $Q_\chi \cong C_{\S_\chi}$.
\end{lemma}
\begin{proof}
We will prove this by showing that the moduli functors $\Q_\chi$ and $\C_{\S_\chi}$ agree on affine schemes.

We start by taking $S = \Spec R$ an affine scheme and $(\F, q) \in \Q_\chi(S)$. By \cref{lemma:fixed_to_equivariant}, it follows that $q\colon \O_{\AA_S^d}^r \to \F$ is a morphism of $T$-equivariant sheaves, and by \cite[Lemma 0EKL]{stacks-project} it follows that it corresponds to a morphism of $\ZZ^d$-graded $R[x_1,\ldots,x_d]$-modules
\[
q \colon R[x_1,\ldots,x_d]^r \to F = \bigoplus_{\vec{a} \in \Zpos^d} F_\vec{a}.
\]
For every $\vec{a} \in \Zpos^d$, this gives a quotient map $q_\vec{a} \colon R^r \to F_{\vec{a}}$. We also know that $\F$ is flat over $R$, which implies that $F_{\vec{a}}$ is flat over $R$ for all $\vec{a} \in \Zpos^d$. By \cite[Lemma 00NX]{stacks-project}, it follows that $F_{\vec{a}}$ is a free $R$-module. Its rank is $\chi(\vec{a})$ by the definition of the characteristic function.

Let $\vec{a},\vec{b} \in \Zpos^d$ be adjacent with $b_k = a_k+1$. We know that $\ker q \subseteq R[x_1,\ldots, x_d]^r$ is an $R[x_1,\ldots,x_d]^r$-module, so in particular it is closed under multiplication by $x_k$. This implies that $\ker q_{\vec{a}} \subseteq \ker q_{\vec{b}}$. In particular, if also $\chi(\vec{a}) = \chi(\vec{b})$, it follows that $F_\vec{a} = F_\vec{b}$. Since each $p \in P_\chi := \bigcup_i P_{i, \chi}$ corresponds to some connected subset of $\Zpos^d$ on which $\chi$ is constant, it follows that we may define the quotient $q_p \colon R^r \to F_p$ such that $q_{\vec{a}} = q_p$ for all $\vec{a} \in p$.

This definition gives that for $i, j \in P_\chi$ with $(i, j) \in I_\chi$, we have $\ker q_i \subseteq \ker q_j$. Note that each quotient $q_i \colon R^r \to F_i$ can also be associated with a quotient of locally free sheaves $q_i \colon \O_S^r \to \F_i$, now it follows that $(\F_i, q_i)_{i \in P_\chi} \in \C_{\S_\chi}(S)$. So we have defined a map from $\Q_\chi(S)$ to $\C_{\S_\chi}(S)$, it can be checked that this map behaves well under base-change $S \to T$ to a different affine scheme.

To define the inverse of this map, let $(\F_i, q_i)_{i \in I} \in \C_{\S_\chi}(S)$. We can reverse the above construction to first get a graded quotient $R$-module
\[
R[x_1, \ldots, x_d]^r \to F = \bigoplus_{\vec{a} \in \Zpos^d} F_\vec{a},
\]
where $F_\vec{a} = F_{[\vec{a}]}$ is the $R$-module corresponding to the sheaf $\F_{[\vec{a}]}$. It follows that $F$ is also an $R[x_1, \ldots, x_d]$-module from the fact that $\ker(q_\vec{a}) \subseteq \ker(q_\vec{b})$ for all adjacent $\vec{a}, \vec{b} \in \Zpos^d$ with $\vec{a} \leq \vec{b}$. Furthermore, each $F_\vec{a}$ is a free $R$-module, which implies that $F$ is flat over $R$. The rank of $F_\vec{a}$ equals exactly $r-d([\vec{a}]) = \chi(\vec{a})$. It therefore corresponds to a $S$-flat, $T$-equivariant quotient $q \colon \O_{\AA_S^d}^r \to \F$, which has characteristic function $\chi$ over any fiber. We get that $(\F, q) \in \Q_\chi(S)$.

We see that the moduli functors of $Q_\chi$ and $C_{\S_\chi}$ agree on all affine schemes, so we conclude that $Q_\chi \cong C_{\S_\chi}$.
\end{proof}

We end this section with a lemma about incidence structures equivalent to $S_\chi$, which will be helpful later on. This basically states that for $\chi \in \X_{r,d}^n$ and $\vec{a}, \vec{b} \in \Zpos^d$ satisfying $\vec{a} \leq \vec{b}$ and $[\vec{a}] \neq [\vec{b}]$, we have that $([\vec{a}], [\vec{b}])$ is contained in the transitive closure of $I_\chi$. In particular, we can add it to $I_\chi$ to get an equivalent incidence structure.

\begin{lemma}\label{lemma:equiv_S_chi}
Let $\chi \in \X_{r,d}^n$. Let $I$ be an anti-reflexive relation on $P_\chi = \bigcup_i P_{i, \chi}$ such that $I_\chi \subseteq I$ and furthermore if $(p, q) \in I$, then there are $\vec{a} \in p$ and $\vec{b} \in q$ such that $\vec{a} \leq \vec{b}$. Then $\S = (P_{1,\chi}, \ldots, P_{r-1,\chi}, I)$ is an incidence structure, and it is equivalent to $\S_\chi$.
\end{lemma}
\begin{proof}
We show that $I$ and $I_\chi$ are relations on $P_\chi$ with the same transitive closure. Since $I_\chi \subseteq I$, it automatically follows that the transitive closure of $I_\chi$ is contained in that of $I$.

For the other direction, let $(p, q) \in I$, so $p \neq q$ and there are $\vec{a} \in p$ and $\vec{b} \in q$ such that $\vec{a} \leq \vec{b}$. Note that, by repeatedly increasing one of the coordinates by 1, we can find a sequence $\vec{a} = \vec{a}_0, \vec{a}_1, \ldots, \vec{a}_k = \vec{b}$ such that for all $0 \leq i < k$ we have that $\vec{a}_i$ and $\vec{a}_{i+1}$ are adjacent, and furthermore $\vec{a}_i \leq \vec{a}_{i+1}$. By \cref{lemma:X} we have that the value $\chi(\vec{a}_i)$ is non-increasing when $i$ increases. So for each $1 \leq i < k$ we either have that $\chi(\vec{a}_i) = \chi(\vec{a}_{i+1})$, in which case $[\vec{a}_i] = [\vec{a}_{i+1}]$, or we have that $\chi(\vec{a}_i) > \chi(\vec{a}_{i+1})$, in which case $([\vec{a}_i], [\vec{a}_{i+1}]) \in I_\chi$. From this we conclude that the pair $(p, q) = ([\vec{a}_0], [\vec{a}_k])$ is contained in the transitive closure of $I_\chi$. This proves that the transitive closure of $I$ is contained in the transitive closure of $I_\chi$.

In particular this implies that $I$ is a subset of $\bigcup_{i < j} P_i \times P_j$, so $\S$ is indeed an incidence structure. Since $I$ and $I_\chi$ have the same transitive closure, it follows that $\S$ and $\S_\chi$ are equivalent.
\end{proof}

\section{Characterizing possible incidence structures} \label{sec:char}
Now that we know that every scheme $Q_\chi$ is isomorphic to some incidence scheme $C_{\S_\chi}$, the next step is to characterize exactly what incidence structures $\S$ are of the form $\S_\chi$ for some $\chi \in \X_{r,d}^\bullet$.

To do this, note that we may identify $\S_\chi$ with an $(r-1)$-partite graph, where each vertex corresponds to some connected subset of $\ZZ^d$, and two vertices are connected by an edge if the corresponding subsets are adjacent. Using \cref{lemma:equiv_S_chi}, we may intuitively expect that not much information is lost if we project these subsets down to the quotient $\ZZ^d/(1, \ldots, 1)$: this could cause some subsets to intersect, but only when they contain $\vec{a}$ respectively $\vec{b}$ with $\vec{a} \leq \vec{b}$ or $\vec{b} \leq \vec{a}$. In this section, we will see that incidence structures coming from some characteristic function can indeed be identified with intersection graphs of subsets of $\RR^{d-1}$.

\begin{definition}
Given a set $U$ and a finite collection of subsets $\set{K_i \subseteq U}{1 \leq i \leq m}$, we define the \emph{intersection graph} of these subsets as the graph with vertex set $\{1, \ldots, m\}$, and with an edge between two vertices $i, j$ if and only if $K_i \cap K_j \neq \emptyset$.
\end{definition}

\begin{theorem} \label{theorem:char_chi}
Let $\S$ be a rank $r-1$ incidence structure. There is some $\chi \in \X_{r,d}^\bullet$ with $\S_\chi$ equivalent to $\S$, if and only if $\S$ is equivalent to an incidence structure corresponding to an $(r-1)$-partite intersection graph of nonempty, path-connected, compact subsets of $\RR^{d-1}$.
\end{theorem}
\begin{proof}
We start by showing that for each $\chi \in \X_{r,d}^\bullet$, the structure $\S_\chi$ is equivalent to some structure corresponding to an intersection graph of the desired type. For every point 
\[
\vec{a} = (a_1,\ldots,a_d) \in \Zpos^d,
\]
we let $B_{\vec{a}} \subseteq \RR^{d-1}$ be the closed ball centered at $(a_1-a_d, a_2-a_d, \ldots, a_{d-1}-a_d)$ with radius~$\frac{2}{3}$. The following two properties hold for this assignment:
\begin{enumerate}
\item If $B_\vec{a} \cap B_\vec{b} \neq \emptyset$ for $\vec{a}, \vec{b} \in \Zpos^2$, then either $\vec{a} \leq \vec{b}$ or $\vec{b} \leq \vec{a}$.
\item If $\vec{a}, \vec{b} \in \Zpos^2$ are adjacent, then $B_\vec{a} \cap B_\vec{b} \neq \emptyset$.
\end{enumerate}
Now to an element $p \in P_\chi$, we can assign the set $K_p := \bigcup B_\vec{a}$ where the union is taken over all $\vec{a} \in \Zpos^d$ in the subset corresponding to $p$. The second property of the balls $B_\vec{a}$, and the fact that $p$ corresponds to a connected subset of $\Zpos^d$, guarantees that this union will be a nonempty, path-connected, compact set. Let $I$ be the relation on $P_\chi$ which contains $(p, q)$ if and only if $p \neq q$, $d(p) \leq d(q)$ and $K_p \cap K_q \neq \emptyset$. Stated differently, $I$ contains exactly all the edges $(p, q)$ of the intersection graph of $\set{K_p}{p \in P_\chi}$, where we direct these edges such that $d(p) \leq d(q)$.

From the stated properties of the sets $B_\vec{a}$, it follows that $I$ contains $I_\chi$, and furthermore, if $(p, q) \in I$, then there are $\vec{a} \in p$ and $\vec{b} \in q$ such that $B_\vec{a} \cap B_\vec{b} \neq \emptyset$, so $\vec{a} \leq \vec{b}$. By \cref{lemma:equiv_S_chi}, we see that $\S = (P_\chi, I)$ is an incidence structure equivalent to $\S_\chi$. So indeed each $\S_\chi$ is equivalent to an incidence structure corresponding to an $(r-1)$-partite intersection graph of nonempty, path-connected, compact subsets of $\RR^{d-1}$.

\bigskip

\begin{figure}
\centering
\parbox{0.48\textwidth}{
\begin{subfigure}[b]{0.48\textwidth}
\includegraphics[width=\textwidth]{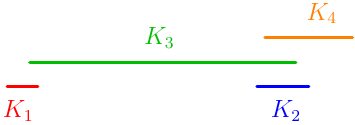}
\caption{The sets $K_p \subseteq \RR$}
\label{fig:picK}
\end{subfigure}

\vspace{1cm}

\begin{subfigure}[b]{0.48\textwidth}
\includegraphics[width=\textwidth]{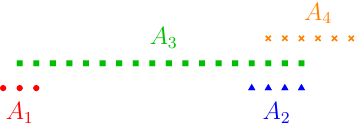}
\caption{The sets $A_p \subseteq \Zpos$}
\label{fig:picA}
\end{subfigure}
}
\hspace{0.02\textwidth}
\parbox{0.48\textwidth}{
\begin{subfigure}[b]{0.48\textwidth}
\includegraphics[width=\textwidth]{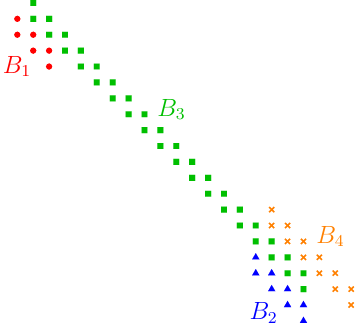}
\caption{The sets $B_p \subseteq \Zpos^2$}
\label{fig:picB}
\end{subfigure}
}

\vspace{1cm}

\parbox{0.48\textwidth}{
	\begin{subfigure}[b]{0.48\textwidth}
		\includegraphics[width=\textwidth]{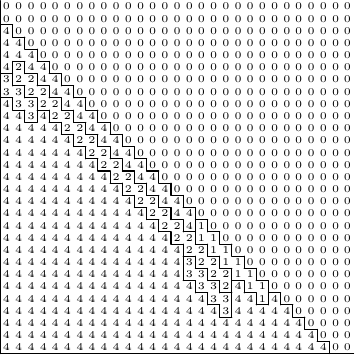}
		\caption{The function $\eta \colon \Zpos^2 \to \ZZ$}
		\label{fig:picEta}
	\end{subfigure}
}
\hspace{0.02\textwidth}
\parbox{0.48\textwidth}{
	\begin{subfigure}[b]{0.48\textwidth}
		\includegraphics[width=\textwidth]{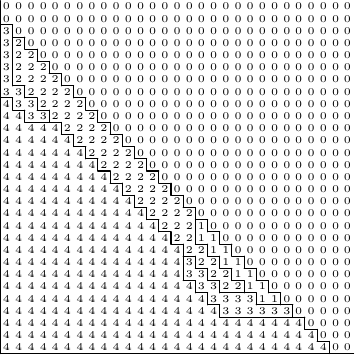}
		\caption{The function $\chi \colon \Zpos^2 \to \ZZ$}
		\label{fig:picChi}
	\end{subfigure}
}
\caption{An example of the construction from the second part of the proof of \cref{theorem:char_chi}, with $d = 2$, $r = 4$, $P_1 = \{1,2\}$, $P_2 = \{3\}$ and $P_3 = \{4\}$}
\label{fig:main_constr}
\end{figure}

Next we show that any incidence structure $\S = (P_0, \ldots, P_{r-2}, I)$ corresponding to some intersection graph of path-connected, compact subsets of $\RR^{d-1}$ is also equivalent to $\S_\chi$ for some $\chi \in \X_{r,d}^\bullet$. An example of this construction is shown in \cref{fig:main_constr}. Let $K_p \subseteq \RR^{d-1}$ for $p \in P$ be the compact set corresponding to $p$, so for $p, q \in P$ with $d(p) < d(q)$ we have $(p, q) \in I$ if and only if $K_p \cap K_q \neq \emptyset$. Furthermore, if $p, q \in P$ satisfy $d(p) = d(q)$, then $K_p \cap K_q = \emptyset$. Without loss of generality, we assume that $K_p \subseteq \RR_{\geq 0}^{d-1}$.

We first want to convert each $K_p$ into a subset of $\Zpos^{d-1}$. To achieve this, let $\delta \in \RR_{>0}$ be such that for any $p, q \in P$, we have that $K_p \cap K_q = \emptyset$ implies that the (Euclidean) distance between $K_p$ and $K_q$ is at least $\delta (4(r-1) + \sqrt{d})$. Now for all $p \in P$ we define
\[
A_p = \set{\vec{a} \in \Zpos^{d-1}}{\left[\delta a_1, \delta(a_1+1)\right]\times \cdots \times \left[\delta a_{d-1}, \delta(a_{d-1}+1)\right] \cap K_p \neq \emptyset},
\]
see also \cref{fig:picA}. Intuitively, we put a grid with cells of size $\delta$ on $\RR_{\geq 0}^{d-1}$, and let $A_p$ denote the indices of all the cells that intersect $K_p$. Note that each $A_p$ is connected, and for $(p, q) \in I$ we have $A_p \cap A_q \neq \emptyset$. Furthermore, if $p, q \in P$ with $d(p) \leq d(q)$ but $(p,q) \not\in I$, then the Euclidean distance between $K_p$ and $K_q$ is at least $\delta (4(r-1) + \sqrt{d})$, and it follows that the Manhattan-distance between any element of $A_p$ and any element of $A_q$ is at least $4(r-1)$.

The next step is to replace the sets $A_p$ by subsets of $\Zpos^d$. Let $M \in \Zpos$ be chosen such that $M > \sum_{k=1}^{d-1} a_k$ for any $\vec{a} \in \bigcup_{p \in P} A_p$. For every $p \in P$, we define the set $B_p \subseteq \Zpos^d$ in the following way:
\begin{multline*}
B_p = \set{\left(a_1, \ldots, a_{d-1}, M+2(d(p)-1) - \sum_{k=1}^{d-1} a_k\right)}{\vec{a} \in A_p} \\
\cup \set{\left(a_1, \ldots, a_{d-1}, M+2(d(p)-1)+1 - \sum_{k=1}^{d-1} a_k\right)}{\vec{a} \in A_p}.
\end{multline*}
See \cref{fig:picB}. This assignment satisfies the following properties, which will be proven after their statement:
\begin{enumerate}
\item Each $B_p$ is connected.
\item The sets $B_p$ are pairwise disjoint.
\item All the sets $B_p$ are contained in $\set{\vec{a} \in \Zpos^d}{M \leq \sum_{k=1}^d a_k < M+2(r-1)}$.
\item For $(p, q) \in I$, there are $\vec{a} \in B_p$ and $\vec{b} \in B_q$ with $\vec{a} \leq \vec{b}$.
\item If $p, q \in P$ with $A_p \cap A_q = \emptyset$, then for any $\vec{a} \in B_p$, $\vec{b} \in B_q$ and $\vec{c} \in \Zpos^d$ with $\vec{a} \leq \vec{c}$ and $\vec{b} \leq \vec{c}$, we have $\sum_{k=1}^d c_k \geq M+2(r-1)$.
\end{enumerate}
The first property here follows from the definition of $B_p$ and the fact that $A_p$ is connected. The second property uses that for $p, q \in P$ with $d(p) = d(q)$ we have $A_p \cap A_q = \emptyset$. Property 3 follows directly from the definition of $B_p$. Property 4 additionally uses that $(p, q) \in I$ if and only if $d(p) \leq d(q)$ and $A_p \cap A_q \neq \emptyset$. The last property uses the fact that the Manhattan-distance between $A_p$ and $A_q$ is at least $4(r-1)$ if $A_p \cap A_q = \emptyset$, so by the triangle inequality the distance between either $\vec{a}$ and $\vec{c}$ or $\vec{b}$ and $\vec{c}$ should be at least $2(r-1)$. Combined with the fact that $\vec{a} \leq \vec{c}$ and $\vec{b} \leq \vec{c}$ this gives the desired inequality.

Using these sets $B_p$, we define a helper function $\eta \colon \Zpos^d \to \ZZ$ as
\[
\eta(\vec{a}) =
\begin{cases}
	r & \mbox{if $\sum_{k=1}^d a_i < M$}\\
	0 & \mbox{if $\sum_{k=1}^d a_i \geq M+2(r-1)$,}\\
	r-d(p) & \mbox{if $\vec{a} \in B_p$ for some $p \in P$,}\\
	r & \mbox{otherwise,}
\end{cases}
\]
for $\vec{a} \in \Zpos^d$. See \cref{fig:picEta}. This is well-defined by properties 2 and 3 of the sets $B_p$. Now we define the characteristic function $\chi$ as
\[
\chi(\vec{a}) = \min_{\vec{b} \leq \vec{a}} \eta(\vec{b}) \qquad \text{for all } \vec{a} \in \Zpos^d.
\]
See \cref{fig:picChi}. From this definition and \cref{lemma:X} it immediately follows that $\chi \in \X_{r,d}^\bullet$, so we can consider the corresponding incidence structure $\S_\chi = (P_\chi, I_\chi)$. It can also be checked, using properties 2, 3 and 5 of the sets $B_p$, that for all $p \in P$ and $\vec{a} \in B_p$, we have $\chi(\vec{a}) = \eta(\vec{a}) = r-d(p)$. From this, and properties 1 and 5, it follows that for any $p \in P$, there is a unique $q \in P_\chi$ such that $B_p$ is a subset of $q$. Denote $C_p = q$. This induces an identification between $P$ and $P_\chi$, by sending $p \in P$ to $C_p \in P_\chi$.

Property 4 of the sets $B_p$ now implies for $(p, q) \in I$, that there are $\vec{a} \in C_p, \vec{b} \in C_q$ such that $\vec{a} \leq \vec{b}$. Furthermore, if $(C_p, C_q) \in I_\chi$, then there are $\vec{c} \in C_p$ and $\vec{d} \in C_q$ with $\vec{c} \leq \vec{d}$. By the definition of $\eta$ and $\chi$ it follows that there are $\vec{a} \in B_p, \vec{b} \in B_q$ with $\vec{a} \leq \vec{c} \leq \vec{d}$ and $\vec{b} \leq \vec{d}$. Since $\vec{d}$ satisfies
\[
\sum_{k=1}^d d_k < M+2(r-1),
\]
it follows using property 5 that $(p,q) \in I$. \Cref{lemma:equiv_S_chi} now shows that $\S$ and $\S_\chi$ are equivalent. This completes the proof of the theorem.
\end{proof}

\section{Applications} \label{sec:appl}

\subsection{Rank 1 and 2} \label{subsec:rank12}
For low rank, we can determine the possibilities for $Q_\chi$ directly from the fact that it is a rank $r-1$ incidence scheme. This gives the following result:
\begin{corollary}
Let $\chi \in \X_{r,d}^n$. If $r = 1$, then $Q_\chi$ is a reduced point. If $r = 2$, then $Q_\chi$ is a product of copies of $\PP^1$. In particular, if $r \leq 2$, then $(Q_{r,d}^n)^T$ is smooth.
\end{corollary}
\begin{proof}
For $r = 1$, note that $Q_\chi$ is a rank 0 incidence scheme by \cref{lemma:Q_is_C}. This immediately implies that $Q_\chi$ consists of just a single point.

For $r = 2$, the scheme $Q_\chi$ is a rank 1 incidence scheme of points in $\PP^1$. Since there cannot be any incidence relations between these points, it follows that $Q_\chi \cong \prod_{i=1}^k \PP^1$, where $k$ is the number of points parametrized by $\S_\chi$.
\end{proof}

Note that the $r = 1$ case here is exactly the well-known result that the torus-fixed locus of the Hilbert scheme of $n$ points $\Hilb_{\AA^d}^n = Q_{1, d}^n$ consists of isolated, reduced points.

\subsection{Dimension 1 and 2} \label{subsec:dim12}
If we consider the case $d = 1$, we see that all incidence structures $\S_\chi$ for $\chi \in \X_{r,1}^n$ correspond to intersection graphs of nonempty subsets of $\RR^0$. Since there is only one such subset, we immediately get the following corollary.

\begin{corollary}
Let $\S = (P_1, \ldots, P_{r-1}, I)$ be an incidence structure. There is some $\chi \in \X_{r,1}^\bullet$ with $\S_\chi$ equivalent to $\S$, if and only if $|P_i| \leq 1$ for all $1 \leq i \leq r-1$ and $\S$ is equivalent to an incidence structure corresponding to a complete graph. If this is the case, then $C_\S$ is isomorphic to the flag variety $\Fl(d_1, \ldots, d_k, r)$, where $d_1 < \cdots < d_k$ are the integers such that $P_{d_i} \neq \emptyset$.
\end{corollary}

When $d = 2$, we get exactly the interval graphs.

\begin{definition}
An \textit{interval graph} is a graph which is the intersection graph of a set of bounded, closed intervals $\set{J_i \subseteq \RR}{1 \leq i \leq m}$.
\end{definition}

\begin{corollary}
Let $\S$ be an incidence structure, and let $r$ be a nonnegative integer. There is some $\chi \in \X_{r,2}^\bullet$ with $\S_\chi$ equivalent to $\S$, if and only if $\S$ is equivalent to an incidence structure corresponding to an $(r-1)$-partite interval graph.
\end{corollary}

In the rest of this section, we will show that the incidence schemes corresponding to interval graphs are always smooth. In particular, this implies that $(Q_{r,2}^\bullet)^T$ is smooth for all $r$. In the proof, we use the following lemma.

\begin{lemma} \label{lemma:flag_proj_smooth}
Let $0 \leq d_1 < d_2 < d_3 \leq r$ be nonnegative integers. Then the projection
\[
\Fl(d_1, d_2, d_3, r) \to \Fl(d_1, d_3, r)
\]
is smooth.
\end{lemma}
\begin{proof}
We will show that $\Fl(d_1, d_2, d_3, r)$ is a $\Gr(d_2-d_1, d_3-d_1)$-bundle over $\Fl(d_1, d_3, r)$. Note that the scheme $\Fl(d_1, d_3, r)$ parametrizes, over some $S \in \Sch_\ZZ$, pairs of quotients $q_1 \colon \O_S^r \to \F_1$ and $q_3 \colon \O_S^r \to \F_3$, where $\F_1$ is locally free of rank $r-d_1$ and $\F_3$ is locally free of rank $r-d_3$, such that furthermore $\ker(q_1) \subseteq \ker(q_3)$. This implies that $q_3$ uniquely factors as $q_3 = q_3' \circ q_1$, where $q_3' \colon \F_1 \to \F_3$ is another quotient map. So $\Fl(d_1, d_3, r)$ can also be considered to parametrize chains of quotients of locally free sheaves
\[
\O_S^r \to \F_1 \to \F_3.
\]
Similarly, $\Fl(d_1, d_2, d_3, r)$ can be considered to parametrize chains of quotients of the form
\[
\O_S^r \to \F_1 \to \F_2 \to \F_3,
\]
where $\F_1, \F_2, \F_3$ are locally free sheaves of rank respectively $r-d_1$, $r-d_2$ and $r-d_3$. For ease of notation, we will denote these ranks by $r_1$, $r_2$ and $r_3$ respectively.
	
Choose embeddings
\[
\O_{\Spec \ZZ}^{r_3} \hookrightarrow \O_{\Spec \ZZ}^{r_1} \hookrightarrow \O_{\Spec \ZZ}^r.
\]
Now consider the open subset $U \subseteq \Fl(d_1, d_3, r)$ which over $S$ parametrizes chains
\[
\O_S^r \to \F_1 \to \F_3
\]
such that the compositions $\O_S^{r_1} \hookrightarrow \O_S^r \to \F_1$ and $\O_S^{r_3} \hookrightarrow \O_S^r \to \F_3$ are isomorphisms. Note that these $U$ cover $\Fl(d_1, d_3, r)$ when varying the embeddings of $\O_{\Spec \ZZ}^{r_1}$ and $\O_{\Spec \ZZ}^{r_3}$.
	
Lifting a point in $U$ corresponding to a chain $\O_S^r \to \F_1 \to \F_3$ to a point in $\Fl(d_1,d_2,d_3,r)$ corresponds to giving a pair of quotient maps $\F_1 \to \F_2$ and $\F_2 \to \F_3$ such that the composition
\[
\O_S^{r_3} \hookrightarrow \O_S^{r_1} \cong \F_1 \to \F_2 \to \F_3 \cong \O_S^{r_3}
\]
is the identity. From this we can get a quotient map $\O_S^{r_1}/\O_S^{r_3} \to \F_2/\O_S^{r_3}$, which gives a point in $\Gr(d_3-d_1, d_2-d_1)(S)$. It can be checked that this gives a bijection
\[
(\Fl(d_1,d_2,d_3,r) \times_{\Fl(d_1,d_2,r)} U)(S) \to \Gr(d_3-d_1, d_2-d_1)(S) \times U(S),
\]
which extends to an equivalence of the corresponding moduli functors. Since we can cover $\Fl(d_1,d_2,r)$ by opens $U$ of this form, it follows that $C_\S$ is indeed a $\Gr(d_2-d_1, d_3-d_1)$-bundle over $\Fl(d_1,d_2,r)$, and therefore the projection
\[
\Fl(d_1, d_2, d_3, r) \to \Fl(d_1, d_3, r)
\]
is smooth.
\end{proof}

\begin{proposition}
For any nonnegative integers $r$ and $n$, the scheme $(Q_{r,2}^n)^T$ is smooth.
\end{proposition}
\begin{proof}
Let $\S = (P, I)$ be an incidence structure corresponding to an $(r-1)$-partite interval graph. Let $\set{J_i \subseteq \RR}{i \in P}$ be the corresponding intervals. We will show that we can drop one element from $P$ to get a new smaller incidence structure $\S'$, with the incidence relations still coming from the same set of intervals, such that the projection $C_\S \to C_{\S'}$ is smooth. This proves by induction that $C_\S$ is smooth, and therefore that for all $\chi \in \X_{r,2}^n$ also $Q_\chi \cong C_{\S_\chi}$ is smooth. This would prove the proposition.

Let $p_2$ be the element of $P$ such that $x := \max J_{p_2}$ is minimized, and denote $d_2 = d(p_2)$. Note that this means that for any $p \in P$ with $J_p \cap J_{p_2} \neq \emptyset$, we have $x \in J_p$. In particular, if $p, q \in P$ are chosen such that $J_p$ and $J_q$ both intersect $J_{p_2}$, then $J_p \cap J_q \neq \emptyset$. Define $P' = P \setminus \{p_2\}$, and let $I' = \set{(p, q) \in I}{p \neq p_2 \wedge q \neq p_2}$. Let $\S' = (P', I')$ be the corresponding incidence structure

Suppose that there is at least one $p \in P$ such that $(p_2, p) \in I$. Then define $p_3 \in P$ to be the element with $(p_2, p_3) \in I$ which minimizes $d(p_3)$, and denote $d_3 = d(p_3)$. By the preceding discussion, it follows that for all other $p \in P$ with $(p_2, p) \in I$, we have $J_p \cap J_{p_2} \neq \emptyset$, so $(p_2, p) \in I$. In the case that there is no $p \in P$ with $(p_2, p) \in I$, we define $d_3 = r$.

Similarly, if it exists, we can define $p_1$ to be the element of $P$ maximizing $d(p_1)$ such that $(p_1, p_2) \in I$, and denote $d_1 = d(p_1)$. If there is no such $p_1$, we instead use $d_1 = 0$. In case that $p_1$ and $p_3$ are both defined, we can define $\tilde{I} = I' \cup \{(p_1, p_2), (p_2, p_3)\}$. If one (or both) of $p_1$ and $p_3$ is not defined, we use the same definition except for the fact that we leave out the corresponding edge of $(p_1, p_2)$ and $(p_2, p_3)$. Denote $\tilde{\S} = (P, \tilde{I})$. Note that $p_1$ and $p_3$ are defined such that $I$ is contained in the transitive closure of $\tilde{I}$, so $\S$ and $\tilde{\S}$ are equivalent incidence structures.

Note that there is a projection map $C_{\tilde{\S}} \cong C_\S \to \Fl(d_1, d_2, d_3, r)$, which only remembers the components of $C_\S$ corresponding to $p_1$, $p_2$ and $p_3$, if those elements exist. Similarly, there is a projection $C_{\S'} \to \Fl(d_1, d_3, r)$. These projections can be seen to sit in a Cartesian square of the following form:
\[
\begin{tikzcd}
C_{\tilde{\S}} \ar[r]\ar[d]& \Fl(d_1, d_2, d_3, r)\ar[d]\\
C_{\S'}\ar[r] & \Fl(d_1, d_3, r)
\end{tikzcd}
\]
By \cref{lemma:flag_proj_smooth} it follows that the projection $\Fl(d_1, d_2, d_3, r) \to \Fl(d_1, d_3, r)$ is smooth, therefore also $C_\S \cong C_{\tilde{\S}} \to C_{\S'}$ is smooth.
\end{proof}

\subsection{Dimension 3} \label{subsec:dim3}
Next we can look at the case $d = 3$. In this case the relevant incidence structures correspond to intersection graphs of nonempty connected compact subsets of $\RR^2$. These graphs are also knows as string graphs.

\begin{definition}
A graph is a \textit{string graph} if it is the intersection graph of some subsets of $\RR^2$ (``strings'') that are all homeomorphic to the closed interval $[0,1]$.
\end{definition}

\begin{corollary}
Let $\S$ be an incidence structure, and let $r$ be a nonnegative integer. There is some $\chi \in \X_{r,3}^\bullet$ with $\S_\chi$ equivalent to $\S$, if and only if $\S$ is equivalent to an incidence structure corresponding to an $(r-1)$-partite string graph.
\end{corollary}
\begin{proof}
We will give a short argument here why string graphs are exactly the intersection graphs of nonempty connected compact subsets of $\RR^2$. Note that any string graph is an intersection graph of nonempty connected compact subsets of $\RR^2$ by definition. For the other direction, suppose that we have an intersection graph corresponding to connected compacts $\{J_i \subseteq \RR^2\}$. Now for each $J_i$ we can choose some open neighborhood $N_i \supset J_i$ such that for all $i, j$ we have $N_i \cap N_j \neq \emptyset$ if and only if $J_i \cap J_j \neq \emptyset$. For each pair $i, j$ with $J_i \cap J_j \neq \emptyset$ we can furthermore choose a point $p_{ij} = p_{ji} \in J_i \cap J_j$, and then for every $J_i$ we can construct a string $S_i \subseteq N_i$ which passes through all the points $p_{ij}$. The intersection graph of these strings is exactly equal to the original intersection graph of the sets $J_i$.
\end{proof}

\begin{figure}
	\centering
	\begin{subfigure}{0.4\textwidth}
		\centering
		\includegraphics[page=1, width=0.7\textwidth]{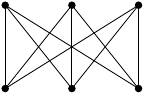}
	\end{subfigure}
	\begin{subfigure}{0.4\textwidth}
		\centering
		\includegraphics[page=2, width=0.7\textwidth]{Visualisation/string_graphs.pdf}
	\end{subfigure}
	\caption{The graph $K_{3, 3}$, and a set of strings which have $K_{3,3}$ as their intersection graph.} \label{fig:K33_string}
\end{figure}

\begin{figure}
	\centering
	\includegraphics[page=3, width=0.3\textwidth]{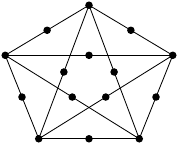}
	\caption{A subdivided $K_5$, which is not a string graph}
	\label{fig:subd_K5}
\end{figure}

String graphs form quite a large class of graphs, for example every planar graph can be seen to be a string graph. There are also string graphs which are not planar graphs, for example the complete bipartite graph on two sets of 3 vertices, $K_{3, 3}$, see \cref{fig:K33_string}. An example of a graph which is not a string graph is the subdivision of $K_5$ shown in \cref{fig:subd_K5}, where every edge is divided into two edges by a vertex.

It turns out that for $d = 3$ and $r \geq 3$, the scheme $(Q_{r,d}^\bullet)^T$ is not smooth.

\begin{example} \label{ex:dim3sing}
Consider the following characteristic function in $\X_{3,3}^{11}$:
\[
\chi(\vec{a}) =
\begin{cases}
	3 & \mbox{if } \vec{a} = (0,0,0),\\
	2 & \mbox{if } \vec{a} \in \{(1,0,0), (0,1,0)\},\\
	1 & \mbox{if } \vec{a} \in \{(1,1,0),(0,0,1),(1,0,1),(0,1,1)\},\\
	0 & \mbox{otherwise.}
\end{cases}
\]
The corresponding incidence structure $\S_\chi$ corresponds to the complete bipartite graph $K_{2,2}$, so the incidence scheme $Q_\chi \cong C_{I_\chi}$ for this characteristic function parametrizes two points and two lines in $\PP^2$, such that both points lie on both lines. This incidence scheme has two irreducible components: one corresponding to configurations where the two points coincide, and one component with configurations where the two lines coincide.

To study the singularity type of this singular locus, consider some embedding of $\AA^2$ in $\PP^2$, and consider the affine open $U \subseteq Q_\chi$ where the points can be given as $(x_1, y_1), (x_1, y_2) \in \AA^2$ and the lines can be given by equations $y = a_1x + b_1$ and $y = a_2x + b_2$, so we can identify
\[
U \cong \Spec\left(\frac{\ZZ[x_1, y_1, x_2, y_2, a_1, b_1, a_2, b_2]}{(y_1-a_1x_1-b_1,y_2-a_1x_2-b_1,y_1-a_2x_1-b_2,y_2-a_2x_2-b_2)}\right).
\]
Note that the rest of $Q_\chi$ can be covered by affine schemes isomorphic to $U$. Now there is an isomorphism
\[
U \to \Spec(\ZZ[x,y]/(xy)) \times \AA^3
\]
which sends
\[
(x_1, y_1, x_2, y_2, a_1, b_1, a_2, b_2) \mapsto (x_1-x_2, a_1-a_2, x_1, a_1, y_1),
\]
so we see that the singularity type of the singular locus of $Q_\chi$ is the same as that of the origin in $\Spec(\ZZ[x,y]/(xy))$.

For higher values of $r$, we can instead consider $\chi \in \X_{r,3}^{8+r}$ given by
\[
\chi(\vec{a}) =
\begin{cases}
	r & \mbox{if } \vec{a} = (0,0,0),\\
	2 & \mbox{if } \vec{a} \in \{(1,0,0), (0,1,0)\},\\
	1 & \mbox{if } \vec{a} \in \{(1,1,0),(0,0,1),(1,0,1),(0,1,1)\},\\
	0 & \mbox{otherwise.}
\end{cases}
\]
For the same reasons as in the $r = 3$ case, the corresponding scheme $Q_\chi$ consists of two irreducible components, and is therefore not smooth.
\end{example}

\subsection{Dimension 4 and higher} \label{subsec:dim4}
In this section, we will prove the main result of this paper, which is that for $r \geq 3$ and $d \geq 4$, the fixed locus $(Q_{r,d}^\bullet)^T$ satisfies Murphy's law. To prove this, we use Mn\"ev-Sturmfels universality.
\begin{theorem}[Mn\"ev-Sturmfels universality]
The disjoint union of all incidence schemes of rank 2 incidence structures satisfies Murphy's law.
\end{theorem}
Proofs of similar statements are given in e.g.~\cite{lafforgue2003chirurgie} and \cite{lee2013mnevsturmfels}. For completeness, here we will show that the result from Lee and Vakil \cite{lee2013mnevsturmfels} implies the formulation given in this paper.

Let $\S = (P_1, P_2, I)$ be a rank 2 incidence structure such that $\{1,2,3,4\} \subseteq P_1$, for every pair of elements $j_1, j_2 \in P_2$ there is an $i \in P_1$ such that $(i, j_1)$ and $(i, j_2)$ are in $I$, and for every $j \in P_2$ there are at least three elements $i \in P_1$ with $(i, j) \in I$. For such an incidence structure, the incidence scheme as defined by Lee and Vakil is the locally closed subscheme
\[
C_\S' \subseteq (\PP^2)^{P_1} \times (\PP^{2\vee})^{P_2}
\]
parametrizing points $(p_i)_{i \in P_1}$ and lines $(\ell_j)_{j \in P_2}$ in $\PP^2$ satisfying the following properties:
\begin{enumerate}
\item $p_1 = (0:0:1), p_2 = (0:1:0), p_3 = (1:0:0), p_4 = (1:1:1)$.
\item For $i \in P_1$ and $j \in P_2$ we have $p_i \in \ell_j$ if and only if $(i, j) \in I$.
\item For $i, j \in P_1$ with $i \neq j$ we have $p_i \neq p_j$ and similarly for $i, j \in P_2$ with $i \neq j$ we have $\ell_i \neq \ell_j$.
\end{enumerate}
It is proven that the disjoint union of all these moduli spaces satisfies Murphy's law.

Note that there is an embedding $C_\S' \times \PGL(3) \to C_\S$ which acts on the moduli functor by sending
\[
((p_i)_{i \in P_1}, (\ell_j)_{j \in P_2}, g) \mapsto ((gp_i)_{i \in P_1}, (g\ell_j)_{j \in P_2}).
\]
It can be seen that this is actually an open embedding, so any singularity type occuring on $C_\S'$ also occurs on $C_\S$. It follows that $C_\S$ satisfies Murphy's law.

We need to slightly extend this result, to also allow for incidence schemes in $\PP^{r-1}$ for higher values in $r$.

\begin{corollary} \label{cor:mnevsturmfels}
Let $r \geq 3$ be an integer. The disjoint union of all incidence schemes of rank $r-1$ incidence structures satisfies Murphy's law.
\end{corollary}
\begin{proof}
Suppose $r \geq 4$, since the case $r = 3$ is handled by the previous Theorem. Let $\S = (P_1, P_2, I)$ be any rank 2 incidence structure. Define $\S' = (P_1, P_2, \ldots, P_{r-1}, I')$ to be the rank $r-1$ incidence structure defined in the following way: For $P_1$ and $P_2$ it uses the same sets as $\S$, furthermore it has $P_3 = \{s\}$, and $P_k = \emptyset$ for $4 \leq k \leq r-1$. Its incidence relations are defined by
\[
I' = I \cup \set{(p, s)}{p \in P_1 \cup P_2}.
\]
We will show that every singularity type occurring on $\S$ also occurs on $\S'$.

Note that there is a projection $C_{\S'} \to \Gr(3, r)$ which only remembers the subspace of $\PP^{r-1}$ corresponding to $s \in P_3$. Consider the standard affine open $U \subseteq \Gr(3,r)$ parametrizing quotients $q \colon \O_S^r \to \F$ such that the composition
\[
\O_S^{r-3} \to \O_S^r \to \F
\]
is an isomorphism, where the first map is the standard embedding into the first $r-3$ coordinates. Note that in this case we also have an isomorphism
\[
\O_S^r \to \F \oplus \O_S^3,
\]
where the first component of this map is given by $q$, and the second component is given by the projection on the last 3 coordinates of $\O_S^3$.

Now consider some point in $(C_\S \times U)(S)$ for some $S \in \Sch_\ZZ$, which is given by quotients $q_i \colon \O_S^3 \to \F_i$ for $i \in P_1 \cup P_2$, and one more quotient $q \colon \O_S^r \to \F$ which induces an isomorphism $\O_S^{r-3} \to \F$. We construct a point in $C_{\S'}$ by assigning to every $i \in P_1 \cup P_2$ the quotient
\[
q_i'\colon \O_S^r \cong \F \oplus \O_S^3 \xrightarrow{\id_\F \oplus q_i} \F \oplus \F_i.
\]
Furthermore, to $s \in P_3$ we assign the quotient $q_s = q$. Note that indeed the kernel of each $q_i'$ for $i \in P_1 \cup P_2$ is contained in $\ker q_s = \ker q$. It can be checked that this assignment gives an isomorphism
\[
C_\S \times U \to C_\S' \times_{\Gr(3,r)} U.
\]
In particular, this means that every singularity type which occurs on $C_\S$ occurs somewhere on this open subscheme of $C_\S'$.
\end{proof}

From this, the main result follows in a fairly straightforward manner. Here we essentially use that any graph can be embedded in $\RR^{d-1}$ for $d$ at least 4, while this is not possible for $d = 3$.

\begin{theorem}
Let $r, d$ be a positive integers with $d \geq 4$. Every rank $r-1$ incidence structure is of the form $\S_\chi$ for some $\chi \in \X_{r, d}^\bullet$. In particular, if $r \geq 3$, the scheme $(Q_{r,d}^\bullet)^T$ satisfies Murphy's law.
\end{theorem}
\begin{proof}
To prove this theorem, it is sufficient to show that every (finite) undirected $(r-1)$-partite graph is an intersection graph of nonempty, path-connected, compact subsets of $\RR^{d-1}$. Then the result immediately follows by \cref{theorem:char_chi} and \cref{cor:mnevsturmfels}.

Let $G = (V, E)$ be an undirected graph. Assign to every $v \in V$ a point $P_v \in \RR^{d-1}$, such that no 4 of these points are coplanar. For $v, w \in V$, denote by $s_{v,w}$ the closed line segment between $P_v$ and the midpoint of $P_v$ and $P_w$. In this way, we have $P_v \in s_{v,w}$ and $s_{v,w} \cap s_{w,v} \neq \emptyset$. Assign to every $v \in V$ the set
\[
K_v = \{P_v\} \cup \bigcup_{\{v, w\} \in E} s_{v,w} \subseteq \RR^{d-1}.
\]
The sets $K_v$ are easily checked to have the desired intersection graph.
\end{proof}

\bibliographystyle{alpha}
\bibliography{references}

\end{document}